\documentclass[12pt,oneside]{amsart}

\usepackage[a4paper]{geometry}  % ,margin=1in

\usepackage{amssymb,mathrsfs}

\usepackage[english]{babel}

\usepackage{color}

\theoremstyle{plain}
 \newtheorem{theorem}{Theorem}[section]
\theoremstyle{definition}
 \newtheorem{definition}[theorem]{Definition}

\theoremstyle{plain}
 \newtheorem{proposition}[theorem]{Proposition}
 \newtheorem{corollary}[theorem]{Corollary}
 \newtheorem{lemma}[theorem]{Lemma} %%Delete [theorem] to re-start numbering
 
 \theoremstyle{definition}
\theoremstyle{remark}
 \newtheorem{remark}[theorem]{Remark}

\def\N{\mathbb{N}}
\def\Z{\mathbb{Z}}

\def\R{\mathbb{R}}

\def\m{\mu}

\DeclareMathOperator{\supp}{supp}
\DeclareMathOperator{\hull}{hull}
\DeclareMathOperator{\sep}{sep}
\DeclareMathOperator{\diam}{diam}

\let\eps\varepsilon

\begin{document}

\title{Coarse medians and Property A}

\author{J\'an \v Spakula}
\address{Mathematical Sciences, University of Southampton, SO17 1BJ, United Kingdom}
\email{jan.spakula@soton.ac.uk}

\author{Nick Wright}
% \address{Mathematical Sciences, University of Southampton, SO17 1BJ, United Kingdom}
\email{n.j.wright@soton.ac.uk}

\begin{abstract}
We prove that uniformly locally finite quasigeodesic coarse median spaces of finite rank and at most exponential growth have Property A. This offers an alternative proof of the fact that mapping class groups have Property A.
\end{abstract}


\maketitle

\section{Introduction}\label{sec:intro}

Coarse median spaces and groups were invented by Bowditch \cite{Bowditch:2013ab,Bowditch:2013aa,Bowditch:2014aa} as (we are guessing here) a device offering a unified approach to hyperbolic groups and mapping class groups.

Indeed, hyperbolic groups are precisely coarse median groups of rank $1$ \cite[Theorem 2.1]{Bowditch:2013ab}, and mapping class groups are instances of coarse median groups of finite rank \cite[Theorem 2.5]{Bowditch:2013ab}.

Furthermore, groups that are relatively hyperbolic with respect to a collection of coarse median groups are again coarse median \cite{Bowditch:2013aa}. This provides more examples of coarse median groups, for instance geometrically finite Kleinian groups and Sela's limit groups.

Coarse median approach to these classes of groups is quite powerful:
In this series of papers, Bowditch uses it to give unified proofs of some properties, for instance the Rapid Decay property, quadratic isoperimetric inequality and a computing the dimension of asymptotic cones.

Intuitively, a coarse median space is a metric space endowed with a ternary structure (a map assigning a point to every triple of points), which is metrically a controlled amount away from being an actual median structure. (Finite) set with an actual median structure are just (vertex sets of) CAT(0) cube complexes. Hence one may loosely regard coarse median structures as coarse versions of (metrized) CAT(0) cube complexes. This analogy works exactly in the ``rank one'' situation, where the CAT(0) cube complexes are trees, and hyperbolic groups are ``coarsely tree-like''. For the actual definitions, see Section \ref{sec:prelim}.

The main result of this piece is that quasigeodesic coarse median spaces of finite rank, which are uniformly locally finite and have at most exponential growth, have Yu's Property A. For proving Property A we use a criterion which is an adaption Brown and Ozawa's proof \cite{Brown:2008qy} that hyperbolic groups act amenably on the boundary. As a side-effect, we obtain a quick proof of Property A for finite dimensional CAT(0) cube complexes, a fact originally established in \cite{Brodzki:2009aa} by a different, more combinatorial, method. Our proof for coarse median spaces is a coarsification of this short argument.

As a consequence, we obtain an alternative proof of the result that mapping class groups have Property A (i.e.~are exact), originally proved by Hamenst\"adt \cite{Hamenstadt:2009aa} and Kida \cite{Kida:2006aa}.

Finally, we would like to mention a related notion of \emph{hierarchically hyperbolic spaces} (and groups), developed recently in \cite{Behrstock:2014aa,Behrstock:2015aa}. While this property is stronger (see \cite[Section 7]{Behrstock:2015aa}), and somewhat more involved than coarse medians it is also substantially more powerful: it implies even finite asymptotic dimension \cite{Behrstock:2015ab}. Having finite asymptotic dimension is a strictly stronger property than Property A. We close off with a question: do coarse median groups of finite rank have finite asymptotic dimension?

\medskip

The structure of the paper is as follows: in Section \ref{sec:prelim} we recall the relevant definitions and facts. Section \ref{sec:criter} explains Brown and Ozawa's criterion for Property A. In Section \ref{sec:cat0} we outline the quick proof of Property A for CAT(0) cube complexes. In Section \ref{sec:topmed} we establish some facts about (metric) median algebras; and finally Section \ref{sec:coarsemed} contains the proof of the main result.

\subsection{Acknowledgements.}

The first author thanks Goulnara Arzhantseva for her encouragement, continuing support, and the initial impetus for this work.

\section{Preliminaries}\label{sec:prelim}

\subsection{CAT(0) cube complexes}\label{subsec:defs-cat0}

We recall the notions related to CAT(0) cube complexes. For details, please consult \cite{Bridson:1999aa,Niblo:1998aa}.

A \emph{cube complex} is a polyhedral complex in which the cells are Euclidean cubes of side length one, the attaching maps are isometries identifying the faces of a given cube with cubes of lower dimension and the intersection of two cubes is a common face of each. One-dimensional cubes are called \emph{edges}; and the complex is \emph{finite-dimensional} if there is a bound on the dimension of its cubes.

Recall that we can endow a cube complex with a naturally defined \emph{geodesic metric}. Furthermore, we can endow the set of vertices of a cube complex with an \emph{edge-path} metric; in the finite-dimensional case, this metric is coarsely equivalent to (the restriction of) the geodesic metric \cite[Proposition 1.7]{Brodzki:2009aa}.

A cube complex is a \emph{CAT(0) cube complex}, if the underlying topological space is simply connected and the complex satisfies Gromov's \emph{link condition} \cite{Gromov:1987aa}.
In the finite-dimensional case, this is equivalent to asking that the geodesic metric satisfies the CAT(0) inequality \cite{Bridson:1999aa}.

A \emph{hyperplane} $H$ (or a \emph{wall}) is a geometric hyperplane, which cuts each cube that it intersects exactly in half. Such an $H$ divides the vertex set into two path-connected subspaces which are referred to as \emph{half-spaces}. Two hyperplanes \emph{cross}, if each of the four possible intersections of the associated half-spaces is non-empty. We say that $H$ \emph{separates} two vertices, if every edge-path connecting them crosses $H$. For two sets for vertices $A$ and $B$, we shall write $A|_HB$ if $H$ separates every vertex in $A$ from every vertex in $B$, i.e. $A$ and $B$ are in different half-spaces determined by $H$. The \emph{interval} $[x,y]$ between two vertices $x,y$ is the intersection of all half-spaces containing both vertices.

Every $n$-dimensional cube in a CAT(0) cube complex defines $n$ pairwise intersecting hyperplanes (whose it \emph{crosses}), and conversely, a collection of $n$ pairwise intersecting hyperplanes define a unique $n$-cube (which crosses exactly these hyperplanes).

Note that the set of vertices of a CAT(0) cube complex is a median algebra in the sense defined below --- the median of three points $x,y,z$ is the unique vertex in the intersection $[x,y]\cap[y,z]\cap[z,x]$ \cite{Roller:1998aa}. Equivalently the median of $x,y,z$ is the unique point lying on a geodesic between $x$ and $y$, on a geodesic between $y$ and $z$ and a geodesic between $z$ and $x$. Furthermore, the notions of an interval, wall, etc\dots are the same whether defined as here, or using the median structure (below).

In a CAT(0) cube complex, each collection of pairwise intersecting hyperplanes determines a unique cube, and conversely, each cube (of dimension $k$) provides $k$ pairwise intersecting hyperplanes.
A \emph{cube path} from a vertex $x$ to a vertex $y$ in a CAT(0) cube complex $X$ is a sequence of cubes $C_0,\dots,C_n$, such that $x$ is a vertex of $C_0$, $y$ is a vertex of $C_n$, and every two consecutive cubes intersect in exactly one vertex.
A \emph{normal cube path} from $x$ to $y$ is a cube path from $x$ to $y$, such that every hyperplane separating $x$ and $y$ is crossed exactly once, with the maximal number of hyperplanes crossed at each step \cite{Niblo:1998aa}. Note that if $X$ is finite-dimensional, then $\frac1d \rho(x,y)\leq n\leq \rho(x,y)$, where $d$ is the dimension of $X$ and $\rho$ denotes the edge-path distance.
We also refer to the sequence of the common vertices between the consecutive cubes on the normal cube path as the normal cube path.

\subsection{Metric median algebras}\label{subsec:defs-metrmed}

We summarise the notions that we need for this paper. For a more thorough account on median structures, we refer to \cite{Bowditch:2014aa,Bowditch:2013ab}. The median algebras can be thought of an abstraction of CAT(0) cube complexes --- every finite median algebra \emph{is} actually the vertex set of a finite CAT(0) cube complex. While in one direction of this link works in general, median algebras can be ``larger'' (for instance $\R$-trees are also median algebras).

A \emph{median algebra} is a set, $\Phi$, equipped with a ternary operation, $\m:\Phi^3\to \Phi$, such that for all $a,b,c,d,e\in \Phi$, we have:
\begin{itemize}
  \item[(M1):] $\m(a,b,c)=\m(b,c,a)=\m(b,a,c)$,
  \item[(M2):] $\m(a,a,b)=a$, and
  \item[(M3):] $\m(a,b,\m(c,d,e))=\m(\m(a,b,c),\m(a,b,d),e)$.
\end{itemize}
While this is the formal definition, we prefer to think about finite median algebras as the vertex sets of finite CAT(0) cube complexes (with the natural median structure).

Given $a,b\in \Phi$, the \emph{interval} $[a,b]$ is defined to by $[a,b]=\{c\in \Phi\mid \m(a,b,c) = c\}$. A subset $H\subset \Phi$ is \emph{convex} if $[a,b]\subset H$ for all $a,b\in H$.

For $A\subset\Phi$, define the \emph{convex hull}, denoted $\hull(A)$, to be the smallest convex subset of $\Phi$ containing $A$. Note that $\hull(\{a,b\})=[a,b]$ for $a,b\in\Phi$. Furthermore, for $A\subset\Phi$, define the \emph{join}, $J(A)=\bigcup_{a,b\in A}[a,b]$. Continuing inductively, we put $J^0(A)=J(A)$ and $J^i(A)=J(J^{i-1}(A))$. In general, there always exists some $p\in\N$ such that $J^p(A)=\hull(A)$, and moreover we know that $p$ can be taken to be no larger than the rank of $\Phi$ \cite[Lemma 5.5]{Bowditch:2013ab}.

A \emph{wall}, $W$, is a partition $\{H^-(W),H^+(W)\}$ of $\Phi$ into two non-empty convex subsets. We say that two walls $W,W'$ \emph{cross} if each of the sets $H^-(W)\cap H^-(W')$, $H^-(W)\cap H^+(W')$, $H^+(W)\cap H^-(W')$ and $H^+(W)\cap H^+(W')$ is non-empty.

We say that $\Phi$ has \emph{rank at most $d$} if there is no collection of $d+1$ pairwise crossing walls of $\Phi$.

By a \emph{topological median algebra} we mean a topological space $\Phi$ endowed with a structure of a median algebra $\m:\Phi^3\to \Phi$, such that $\m$ is continuous in the induced topology. When the topology on $\Phi$ comes from a metric $\rho$, we say that $\Phi$ is a \emph{metric median algebra}.

Let $\Phi$ be a metric median algebra as above. We also recall one of the conditions used in \cite{Bowditch:2014aa} to obtain the embedding result:
\begin{itemize}
  \item[(L2)]: There exists $K\geq 1$ such that for all $a,b,c,d\in\Phi$, $\rho(\m(a,b,c),\m(a,b,d))\leq K\rho(c,d)$.
\end{itemize}
Let us mention that the main embedding result of \cite{Bowditch:2014aa} states that if a metric median algebra $\Phi$ satisfies (L2), is Lipschitz path connected and is $\nu$-colourable\footnote{This is a more restrictive version of rank, which is equivalent to rank for intervals.}, then it bilipschitzly embeds into a product of $\nu$ $\R$-trees.

\subsection{Coarse median spaces}\label{subsec:defs-coarsemed}

In this subsection, we recall the definitions and facts related to coarse medians. For more details, we refer to \cite{Bowditch:2014aa,Bowditch:2013ab}.

Let $(X,\rho)$ be a metric space, and let $\m:X^3\to X$ be a ternary operation. We say that $\m$ is a \emph{coarse median} and that $(X,\rho,\m)$ is a \emph{coarse median space}, if the following conditions hold:
\begin{itemize}
  \item[(C1):] There are constants $K\geq1,H(0)\geq0$, such that for all $a,b,c,a',b',c'\in X$ we have
  $$
  \rho(\m(a,b,c),\m(a',b',c'))\leq K(\rho(a,a')+\rho(b,b')+\rho(c,c')) + H(0).
  $$
  \item[(C2):] There is a function $H:\N\to[0,\infty)$ with the following property. Suppose that $A\subseteq X$ with $1\leq|A|\leq p<\infty$. Then there is a finite median algebra $(\Pi,\m_\Pi)$ and maps $\pi:A\to \Pi$ and $\sigma:\Pi\to X$ such that for all $x,y,z\in \Pi$ we have
  $$
  \rho(\sigma\m_\Pi(x,y,z),\m(\sigma x,\sigma y,\sigma z)) \leq H(p)
  $$
  and
  $$
  \rho(a,\sigma\pi a)\leq H(p)
  $$
  for all $a\in A$.
\end{itemize}
We refer to $K,H$ as the \emph{parameters} of $(X,\rho,\m)$.

Without loss of generality, we may assume that $\m$ satisfies the axioms (M1) and (M2) by \cite[page 22]{Bowditch:2013ab}.

We say that $X$ has \emph{rank at most $d$} if we can always choose $\Pi$ to have rank at most $d$.

Let us recall the asymptotic cones from \cite[Section 9]{Bowditch:2013ab} and \cite[Section 8]{Bowditch:2014aa}: Let $(X,\rho,\m)$ be a coarse median space, let $(r_n)$ be a sequence of positive reals such that $r_n\to\infty$, $(x_n)\subset X$ be a sequence of points in $X$, and finally fix a non-principal ultrafilter on $\N$. With this data, we can construct an ultralimit $(X_\infty,\rho_\infty,\m_\infty)$ of pointed coarse median spaces $((X,\rho/r_n,\m),x_n)$. This ultralimit is referred to as an \emph{asymptotic cone} of $X$ (with the given data), and it is a complete\footnote{The completeness here refers to the metric.} metric median algebra satisfying (L2) (with the constant $K$ being the same as in the definition of coarse median). Moreover if $X$ has rank at most $d$, then $X_\infty$ also has a rank at most $d$.

\subsection{Property A}\label{subsec:def-A}

Property A is coarse geometric property of metric spaces (or more generally coarse spaces), first defined by G.~Yu \cite{Yu:2000aa} as a criterion that (for discrete countable groups) implies the Coarse Baum-Connes conjecture, and hence the Novikov conjecture. The catchphrase here is ``non-equivariant amenability'' or ``coarse amenability''. Since its inception, many equivalent formulations were discovered, including analytic (exactness of the reduced group C*-algebra, nuclearity of the uniform Roe algebra \cite{Guentner:2002aa,Ozawa:2000th}) and dynamical (admitting an amenable action on a compact topological space \cite{Higson:2000dp}).

We shall recall one of the possible definitions (the one used in Proposition \ref{prop:criterion}) for completeness; we refer to \cite{Willett:2009rt} for the whole spectrum.

Let $(X,\rho)$ be a uniformly locally finite discrete metric space. We say that $X$ has \emph{Property A}, if for all $R,\eps>0$ there exists a map $\xi: X\to \ell^1(X)$ from $X$ into the Banach space $\ell^1(X)$, such that
\begin{itemize}
\item $\|\xi(x)\|_1 = 1$ for all $x\in X$;
\item for all $x,y\in X$ with $\rho(x,y)\leq R$, we have $\|\xi(x)-\xi(y)\|_1\leq \eps$;
\item there exists $S>0$, such that for each $x\in X$, $\xi(x)$ is supported in the closed ball $B(x;S)$ around $x$ with radius $S$.
\end{itemize}

\subsection{Geodesicity}\label{subsec:geod}

We shall say that a metric space $(X,\rho)$ is \emph{quasigeodesic}, if there exist constants $G_1,G_2$, such that there exists a \emph{$(G_1,G_2)$-quasigeodesic} between any pair of points in $X$. Note that when $X$ is a quasigeodesic coarse median space, all its asymptotic cones are Lipschitz path connected. This is required for applying the embedding result of Bowditch \cite{Bowditch:2014aa} (and a blanket assumption in \cite{Bowditch:2014aa,Bowditch:2013ab}).

\section{A criterion}\label{sec:criter}

We extract a criterion from a proof of Brown and Ozawa \cite[Theorem 5.3.15]{Brown:2008qy} for proving Property A. Its proof is just an excerpt from \cite{Brown:2008qy}; which is in turn inspired by \cite{Kaimanovich:2004aa}.

\begin{proposition}\label{prop:criterion}
Let $X$ be a uniformly finite, discrete metric space. Suppose that we have an assignment of a set $S(x,k,l)\subset X$ to every $l\in\N$, $k\in\{1,\dots,3l\}$ and $x\in X$, such that:
  \begin{itemize}
    \item[(i)] For every $l\in\N$ there exists $S_l>0$, such that $S(x,k,l)\subset B(x,S_l)$ for all $x\in X$ and $k\in\{1,\dots,3l\}$.
    \item[(ii)] For every $x,y\in X$, $l\geq \rho(x,y)$, $k\in\{l+1,\dots,2l\}$, we have inclusions $S(x,k-\rho(x,y),l)\subset S(x,k,l)\cap S(y,k,l)$ and $S(x,k,l)\cup S(y,k,l)\subset S(x,k+\rho(x,y),l)$.
    \item[(iii)] There exists a function $p$, such that $|S(x,k,l)|\leq p(l)$ for every $x\in X$, $l\in\N$ and $k\in\{1,\dots,3l\}$, satisfying $\lim_{n\to\infty}p(n)^{1/n}=1$.
  \end{itemize}
  Then $X$ has property A.
\end{proposition}

To have some mental picture, let us recall that Brown and Ozawa apply this criterion to hyperbolic groups $\Gamma$, defining the sets as follows: fix a point $u\in\partial\Gamma$. Given $x,k,l$, the set $S(x,k,l)$ consists of points that are exactly $3l$ steps along a geodesic between a point within the $k$-ball around $x$ and $u$. With this definition, the conditions (i) and (ii) follow from the triangle inequality, and (iii) uses the stability of geodesics in hyperbolic spaces (in this case $p$ can be taken to be a linear function).

\begin{proof}
  Consider the Banach space $\ell^1(X)$ and for $A\subset X$ denote by $\chi_A\in\ell^1(X)$ the normalised characteristic function of $A$. Given $n\in\N$ and $x\in X$, define
  $$
  \xi_n(x) = \frac1n\sum_{k=n+1}^{2n}\chi_{S(x,k,n)}
  $$
  Note that $\|\xi_n(x)\|=1$ and $\supp(\xi_n(x)) \subset B(x,S_n)$ for all $x\in X$ by (i).

  To establish Property A, we use the formulation from \cite[Thm 1.2.4/2.]{Willett:2009rt}, recalled also in subsection \ref{subsec:def-A}. We need to show that for a fixed $m$, we have
  $$
  \lim_{n\to\infty}\sup_{\rho(x,y)=m}\|\xi_n(x)-\xi_n(y)\| = 0.
  $$
  First, observe that for any $A,B\subset X$, we have
  $$
  \|\chi_{A} - \chi_{B}\| = 2\left(1 - \frac{|A\cap B|}{\max\{|A|,|B|\}}\right) \leq 2\left(1- \frac{|A\cap B|}{|A\cup B|}\right).
  $$
  Take $x,y\in X$ with $\rho(x,y)=m$ and assume that $n\geq 2m$. Then for any $k\in\{n+1,\dots,2n\}$, applying (ii),
  $$
  \|\chi_{S(x,k,n)}-\chi_{S(y,k,n)}\| \leq 2\left(1- \frac{|S(x,k-m,n)|}{|S(x,k+m,n)|}\right).
  $$
  Consequently
  \begin{align*}
    \|\xi_n(x)-\xi_n(y)\| & \leq \frac1n\sum_{k=n+1}^{2n}\|\chi_{S(x,k,n)}-\chi_{S(y,k,n)}\|\\
      & \leq 2\left(1-\frac1n\sum_{k=n+1}^{2n}\frac{|S(x,k-m,n)|}{|S(x,k+m,n)|}\right)\\
      & \leq 2\left(1-\left(\prod_{k=n+1}^{2n}\frac{|S(x,k-m,n)|}{|S(x,k+m,n)|}\right)^{1/n}\right)\\
      & = 2\left(1-\left(\frac{\prod_{j=n+1-m}^{n+m}|S(x,j,n)|}{\prod_{j=2n+1-m}^{2n+m}|S(x,j,n)|}\right)^{1/n}\right)\\
      & \leq 2\left(1-p(n)^{-2m/n}\right).
  \end{align*}
  We have used the inequality between the arithmetic and geometric mean in the middle step, magic cancellation of many terms in the penultimate step and the last step uses (iii) plus a simple estimate of the sizes of sets by $1$ from below.
  By (iii), the last expression converges to $0$ as $n$ converges to $\infty$. We are done.
\end{proof}

\begin{remark}\label{rem:l-vs-k}
  In the condition (iii), we ask for a bound in terms of $l$. However, it is apparent from the proof that a bound in terms of $k$ with analogous property also suffices.
\end{remark}

\begin{remark}\label{lem:criterion-l-subseq}
  It is clear from the proof of the Proposition that we only need to define the sets $S(x,k,l)$ only for an infinite sequence of indices $l$ (and the corresponding $k\in\{1,\dots,3l\}$), not necessarily for all $l\in\N$.
\end{remark}

\section{CAT(0) cube complexes}\label{sec:cat0}

Proposition \ref{prop:criterion} allows us to quickly prove that finite dimensional CAT(0) cube complexes have property A. This was first proved by Brodzki et.\ al.\ \cite{Brodzki:2009aa} using a more combinatorial approach.

\begin{proposition}\label{prop:cat0cube}
  Let $X$ be a finite dimensional CAT(0) cube complex. Then $X$ has property A.
\end{proposition}

\begin{proof}
  Fix a base vertex $x_0\in X$. Given a vertex $x\in X$, $l\in\N$ and $k\in\{1,\dots,3l\}$, consider the normal cube path from $y$ to $x_0$, where $\rho(y,x)\leq k$. We define the set $S(x,k,l)$ to contain the $3l$-th vertex on such a normal cube path (or $x_0$ if we run out of space). Note that the conditions (i) and (ii) from Proposition \ref{prop:criterion} are automatically satisfied, courtesy of the triangle inequality. To be more precise, if $z\in S(x,k,l)$, then $\rho(x,z)\leq 6ld$, where $d=\dim(X)$.

  To prove the condition (iii) of Proposition \ref{prop:criterion}, we shall argue that if $z\in S(x,k,l)$, then $z\in [x,x_0]$. Or, equivalently:
  
  \textbf{Claim:} Every half-space containing both $x$ and $x_0$ contains also $z$.

  Each hyperplane that we need to consider (i.e. such that one of the associated half-spaces contains both $x$ and $x_0$) either separates $x,x_0$ from $y$ or it does not. In the latter case, the same half-space also clearly contains $z$, so it remains to deal with the former case.

  Denote by $C_0,C_1,\dots,C_m$ the normal cube path from $y$ to $x_0$, and denote by $y=v_0,v_1,\dots, v_m=x_0$ the vertices on this cube path. We shall argue that any hyperplane separating $y$ from $x,x_0$ is ``used'' within the first $\rho(x,y)$ steps on the cube path. Suppose that the cube $C_i$ does not cross any hyperplane $H$ with $\{y\}|_H\{x,x_0\}$. Hence every hyperplane $K$ crossing $C_i$ satisfies $\{y,x\}|_K\{x_0,v_{i+1}\}$. If there was a hyperplane $L$ with $\{y\}|_L\{x,x_0\}$ which was not ``used'' before $C_i$ on the cube path, then necessarily $\{y,v_{i+1}\}|_L\{x,x_0\}$, hence $L$ crosses all the hyperplanes $K$ crossing $C_i$. This contradicts the maximality of this step on the normal cube path. Thus there is no such $L$, and so all the hyperplanes $H$ with $\{y\}|_H\{x,x_0\}$ must be crossed within the first $\rho(x,y)$ steps (as there is at most $\rho(x,y)$ of such hyperplanes).

  Since $z$ is the $3l$-th vertex on the cube path and $\rho(x,y)\leq k\leq 3l$, all the hyperplanes $H$ with $\{y\}|_H\{x,x_0\}$ must have been crossed before $z$. Thus any such $H$ actually also satisfies $\{y\}|_H\{x,x_0,z\}$. We have proved our claim.

  Coming back to showing the condition (iii) of Proposition \ref{prop:criterion}, observe that the interval $[x,x_0]$ embeds isometrically into the cube complex $\R^d$ \cite[Theorem 1.14]{Brodzki:2009aa}, hence it has polynomial growth. This means that as $S(x,k,l) \subset [x,x_0]\cap B(x,6ld)$, its cardinality is bounded by a polynomial in $l$ (of degree $d$). This finishes the proof.
\end{proof}

For the record, we note that dropping the finite dimensionality assumption renders the statement false, namely infinite dimensional CAT(0) cube complexes do not have Property A; this follows from \cite{Nowak:2007aa}, as they contain isometric copies of $(\Z/2\Z)^n$ for arbitrarily large $n$.

\section{Median algebras}\label{sec:topmed}

\begin{definition}\label{def:mu-symbol}
  Let $\Phi$ be a median algebra. Let $n\geq 2$ and $x_1,\dots,x_n,b\in \Phi$. Define
  $$
  \m(x_1;b) := x_1
  $$
  and inductively for $1\leq k < n-1$
  $$
  \m(x_1,\dots,x_{k+1};b) := \m(\m(x_1,\dots,x_{k};b), x_{k+1},b).
  $$
  Note that this definition ``agrees'' with the original median map $\mu$, since $\mu(x_1,x_2;b) = \m(x_1,x_2,b)$.
\end{definition}

Intuitively, $\m(x_1,\dots,x_n;b)$ should be thought of as a projection of $b$ onto the $\hull\{x_1,\dots,x_n\}$, just as $\m(x_1,x_2,b)$ is the projection of $b$ onto $[x_1,x_2]$. However, we do not prove this in this note (but see Lemma \ref{lem:mu-properties}).

\begin{lemma}\label{lem:mu-symmetric}
  The $\m$ symbol from Definition \ref{def:mu-symbol} is symmetric in $x_1,\dots,x_n$.
\end{lemma}

\begin{proof}
  Recalling that interchanging the points in $\m(\cdot,\cdot,\cdot)$ is one of the axioms of a median algebra, it is clearly sufficient to prove that for $n=3$, we can switch $x_2$ with $x_3$. However, applying axioms of median algebras, we have that
  \begin{align*}
    \m(x_1,x_2,x_3;b)
    &=\m(\m(x_1,x_2,b),x_3,b)\\
    &= \m(\m(x_3,b,x_1),\m(x_3,b,b), x_2)\\
    &= \m(\m(x_1,x_3,b),x_2,b)\\
    &= \m(x_1,x_3,x_2;b).
  \end{align*}
  We are done.
\end{proof}

In fact, it is easy to see that $[x_1,b]\cap[x_2,b] = [\m(x_1,x_2,b),b)]$ and then by induction that $\bigcap_{k=1}^n[x_k,b] = [\m(x_1,\dots,x_n;b),b]$.

\begin{lemma}\label{lem:mu-properties}
  Let $\Phi$ be a median algebra. Let $x_1,\dots,x_n,b\in\Phi$.
  \begin{itemize}
    \item[(i)] A wall separates $b$ and $x_1,\dots,x_n$ if and only if it separates $b$ and $\m(x_1,\dots,x_n;b)$.
    \item[(ii)] If $\m(x_1,\dots,x_{n-1};b) \not= \m(x_1,\dots,x_n;b)$ then there exists a wall separating $x_1,\dots,x_{n-1}$ from $x_n,b$.
    \item[(iii)] If, in addition, $\Phi$ has rank at most $d$, then there exists a subset $\{y_1,\dots,y_k\}\subseteq \{x_1,\dots,x_n\}$ with $k\leq d$, such that $\m(y_1,\dots,y_k;b)=\m(x_1,\dots,x_n;b)$.
    \item[(iv)] Assume that $a\in\Phi$ and $x_1,\dots,x_n\in[a,b]$. Then $\{x_1,\dots,x_n\}\subset [a,\m(x_1,\dots,x_n;b)]$.
  \end{itemize}
  \end{lemma}

\begin{proof}
  Since $\m(x,y,b)\in [x,y] = J(\{x,y\})$, we can easily prove by induction that $\m(x_1,\dots,x_n;b)\in J^{n-1}(\{x_1,\dots,x_n\})\subset\hull\{x_1,\dots,x_n\}$. The statement ``$\Longrightarrow$'' in (i) follows. For the converse, assume for contradiction that there exists a wall $W$ that separates $b$ and $\m(x_1,\dots,x_n;b)$ but does not separate $b$ from (say, using Lemma \ref{lem:mu-symmetric}) $x_n$. As half-spaces are convex, the whole interval $[b,x_n]$ is in the same half-space than $b$. But as $\m(\cdot, x_n,b)\in [b,x_n]$, this contradicts the assumption that $b$ separates $\m(x_1,\dots,x_n;b)=\m(\m(x_1,\dots,x_{n-1};b),x_n,b)$.

  For (ii), denote $c=\m(x_1,\dots,x_{n-1};b)$ and note that $\m(x_1,\dots,x_n;b)=\m(c,x_n,b)$. Hence $c\not=\m(x_1,\dots,x_n;b)$ implies $c\not\in[x_n,b]$. As $\{c\}$ and $[x_n,b]$ are convex, this implies, by \cite[Lemma 6.1]{Bowditch:2013ab}, that there is a wall separating $c$ and $x_n,b$. By (i), this wall separates $x_1,\dots,x_{n-1}$ and $x_n,b$.

  For (iii), we proceed by contradiction. Assume that there are at least $d+1$ points in $\{x_1,\dots,x_n\}$ which cannot be removed from the expression $\m(x_1,\dots,x_n;b)$ without changing the result. The previous part of the lemma, together with Lemma \ref{lem:mu-symmetric}, implies that there exist at least $d+1$ different walls which all intersect (for instance, if one wall separates $x_1,\dots,x_{n-1}$ and $x_n,b$, and another one separates $x_1,\dots,x_{n-2},x_n$ and $x_{n-1},b$, they clearly intersect). This contradicts the rank assumption (see \cite[Proposition 6.2]{Bowditch:2013ab}).

  The part (iv) follows by induction from the following statement: if $x,y\in[a,b]$, then $x,y\in[a,\m(x,y,b)]$ (so in particular $[a,x]\subset [a,\m(x,y,b)]$). It is of course sufficient to prove that $x\in [a,\m(x,y,b)]$, which is done using median axioms as follows:
  \begin{align*}
    \m(a,x,\m(x,y,b)) &=
    \m(\m(a,x,x),\m(a,x,b),y)
    = \m(x,x,y) = x.
  \end{align*}
  We are done.
\end{proof}

\begin{lemma}\label{lem:switch-in-int}
  Let $\Phi$ be a median algebra and let $a,b,x,y\in \Phi$ satisfy $x,y\in[a,b]$. Then $y\in[a,x]$ implies $x\in[y,b]$.
\end{lemma}

\begin{proof}
  Using the median axioms and lemma's assumptions, we have
  $$
  \m(y,b,x) = \m(\m(a,x,y),b,x) = \m(\m(b,x,a),\m(b,x,x),y) = \m(x,x,y) = x.
  $$
\end{proof}

\begin{proposition}\label{prop:top-hull}
  Let $\Phi$ be a topological median algebra of rank at most $d$. Given an interval $[a,b]\subset\Phi$ and a compact set $C\subset[a,b]$, there exists $h_1,\dots,h_d\in C$, such that $C\subset [a,\m(h_1,\dots,h_d;b)]$.

  If $\Phi$ is moreover a metric median algebra satisfying the condition (L2), then we have $\rho(a,\m(h_1,\dots,h_d;b))\leq 3^dK^d\max_{1\leq i\leq d}\rho(a,h_i) \leq 3^dK^d\sup_{h\in C}\rho(a,h)$.
\end{proposition}

\begin{proof}
  % \red{First part by compactness, e.g. like this (too complicated?). In the end it's only used for finite $C$, so maybe there's a quicker proof. But I kinda like this one.}
  Consider the compact space $C^d$. Given a tuple $\xi\in C^d$, write $\m(\xi;b)$ for the short. Given $\xi\in C^d$, define
  $$
  A_{\xi} = \left\{\eta\in C^d\mid \m(\eta;b)\in [\m(\xi;b),b] \right\} =
   \left\{ \eta\in C^d \mid [a,\m(\xi;b)]\subset[a,\m(\eta;b)] \right\}.
  $$
  Note that the two conditions are equivalent: if $\m(\eta;b)\in[\m(\xi;b),b]$, then by Lemma \ref{lem:switch-in-int} $\m(\xi;b)\in[a,\m(\eta;b)]$, hence $[a,\m(\xi;b)]\subset [a,\m(\eta;b)]$. Conversely, the last inclusion implies $\m(\xi;b)\in [a,\m(\eta;b)]$, so again by Lemma \ref{lem:switch-in-int} we have $\m(\eta;b)\in[\m(\xi;b),b]$.

  Observe that each set $A_{\xi}$ is closed: In fact it is exactly the inverse image of the closed\footnote{Intervals are closed in topological median algebras; this just uses continuity of $\mu$.} set $[\m(\xi;b),b]\subset \Phi$ under the continuous map $\m(\cdot;b): C^d \to\Phi$.

  Finally, the collection $\{A_{\xi}\mid \xi\in C^d\}$ of subsets of $C^d$ has the finite intersection property. Given $\xi,\eta\in C^d$, by Lemma \ref{lem:mu-properties} (iii), there is $\omega\in C^d$, such that $\m(\omega;b)=\m(\xi\cup\eta;b)$. However, from the definition of $\mu$ and Lemma \ref{lem:mu-symmetric}, we know that $\m(\xi\cup\eta;b)$ belongs to both $[\m(\xi;b),b]$ and $[\m(\eta;b),b]$. In other words, $\omega \in A_{\xi}\cap A_{\eta}$.

  Now, as $C^d$ is compact, there exists $\zeta\in\bigcap_{\xi\in C^d}A_{\xi}$. This means that $[a,\m(\xi;b)]\subset [a,\m(\zeta;b)]$ for all $\xi\in C^d$. In particular, by Lemma \ref{lem:mu-properties} (iv), $\xi\subset [a,\m(\xi;b)]\subset [a,\m(\zeta;b)]$ for all $\xi\in C^d$, so $[a,\m(\zeta;b)]$ contains all the points of $C$. Now just enumerate $\zeta$ as $h_1,\dots,h_d$.

  For the second part of the proposition, we do inductive estimates using (L2). Denote $T=\max_{1\leq i\leq d}\rho(a,h_i)$. Then, as the first step, $\rho(a,\m(h_1;b))=\rho(a,h_1)\leq T$. We show inductively that $\rho(a,\m(h_1,\dots,h_i;b))\leq 3^{i-1}K^{i-1}T$. Assuming this inequality for $i$, denoting $\m(h_1,\dots,h_i;b)=g_i$ we estimate %\red{quite stupidly}
  \begin{align*}
    \rho(a,\m(h_1,\dots,h_{i+1};b)) &= \rho(a,\m(g_i,h_{i+1};b))\\
    &\leq \rho(a,g_i) + \rho(\m(g_i,g_i,b),\m(g_i,h_{i+1},b))\\
    &\leq 3^{i-1}K^{i-1}T + K\rho(g_i,h_{i+1})\\
    &\leq 3^{i-1}K^{i-1}T + K\left(3^{i-1}K^{i-1}T + T\right)\\
    &= T\left(3^{i-1}K^i + 3^{i-1}K^{i-1} + K\right) \leq 3^iK^iT.
  \end{align*}
  We are done.
\end{proof}

% \red{I suppose the main point can be proved more easily using embedding into $\R^d$, and we're going to need this anyway to prove A. But I like doing this by hand.}

\section{Coarse medians}\label{sec:coarsemed}

We shall adapt the idea of `moving deep into the interval' from the CAT(0) cube complex setting to the more general coarse median spaces.

To explain the idea, consider two points $a,b$ and the context--appropriate notion of the interval $[a,b]$. In the CAT(0) cube complex case, we have moved deep into this interval by stepping sufficiently far along the cube path from $a$ to $b$. In the coarse median case we shall be, roughly speaking, looking for `the other end' of the convex hull of $B(a,l)\cap[a,b]$ (cf.\ the second bullet in Corollary \ref{cor:coarse-hull}). Following the suit of \cite{Bowditch:2014aa}, this is done by going to the asymptotic cone (where the results of Section \ref{sec:topmed} can be applied).

\smallskip

We begin by fixing a fair amount of notation.

For the rest of this section, when we say that $X$ is a coarse median space, we mean that $X$ is a coarse median space, with metric denoted $\rho$, the median function denoted $\m$, with parameters $K$ and $H$. These will be fixed throughout.

When convenient, we shall be using the notation $x\sim_s y$ for $\rho(x,y)\leq s$.

Given $\tau\geq0$ and $a,b\in X$, we shall denote by $[a,b]_\tau$ the coarse interval between $a$ and $b$, i.e.\ $[a,b]_\tau := \{ x\in X\mid \mu(a,b,x)\sim_\tau x\}$.

We will denote by $\lambda\geq0$ a constant (depending only on $K$ and $H$), such that for all $x,y,z\in X$ we have $\m(x,y,z)\in[x,y]_\lambda$; its existence is proved in \cite[Lemma 9.2]{Bowditch:2014aa}.

Recall that since the median axiom (M3) holds exactly in median algebras, it does hold in coarse median spaces up to a constant $\gamma\geq0$, depending only on the parameters $K, H$ of the coarse median structure (actually $\gamma=3K(3K+2)H(5)+(3K+2)H(0)$). By this we mean that $\mu(x,y,\mu(z,v,w))\sim_\gamma \mu(\mu(x,y,z),\mu(x,y,v), w)$ for all $x,y,z,v,w\in X$. We shall be using $\gamma$ and this fact throughout this section.

Fixing some more notation, given $r,t\geq0$, denote
\begin{align*}
L_1(r)&=(K+1)r+K\lambda+\gamma+2H(0),\\
L_2(r)&=(K+2)r+H(0), \text{ and}\\
L_3(r,t)&=3^dK^drt+r. 
\end{align*}
The point is that $L_1$ and $L_2$ are linear functions of $r$, and $L_3$ is linear in $r$ with $t$ fixed, and bounded by a linear function of $rt$ (for $t\geq 1$).

\begin{lemma}\label{lem:containment-int}
  Let $X$ be a coarse median space, $r\geq0$, and let $a,b\in X$, $x\in[a,b]_\lambda$. Then $[a,x]_r\subset[a,b]_{L_1(r)}$.
\end{lemma}

\begin{proof}
  Let $z\in[a,x]_r$. Thus $\m(a,x,z)\sim_r z$ and by assumption also $\m(a,b,x)\sim_\lambda x$. Hence
  \begin{align*}
    \m(a,b,z) &\sim_{Kr+H(0)} \m(a,b,\m(a,x,z))\\
    &\sim_\gamma \m(\m(a,b,a),\m(a,b,x),z)\\
    &\sim_{K\lambda+H(0)} \m(a,x,z) \sim_{r} z.
  \end{align*}
  Thus $\rho(\m(a,b,z),z)\leq (K+1)r+K\lambda+\gamma+2H(0)=L_1(r)$, which means by definition that $z\in[a,b]_{L_1(r)}$.
\end{proof}

In what follows, $r$ can be thought of `a scale' and $t$ of `a distance'. In other words, the statements can read as `given a distance ($t$) at which we want the space to behave, there exists a scale ($r_t)$, such that on all larger scales ($r\geq r_t$) it behaves as a median space, with an error proportional to $r$'.

\begin{proposition}\label{prop:coarse-hull}
  Let $X$ be a quasigeodesic coarse median space of rank at most $d$.
  For every $\kappa>0$ and $t>0$, there exists $r_t>0$, such that for all $r\geq r_t$, $a,b\in X$ and $A\subset B(a,rt)\cap [a,b]_\kappa$ with $\sep(A)\geq r$, there exists $h\in [a,b]_{L_1(r)}$, such that
  \begin{itemize}
    \item $\rho(a,h)\leq L_3(r,t)$ and
    % \item for all $x\in B(a,rt)$ we have $\rho(x,\m(a,b,x))\leq \kappa$ implies $\rho(x,\m(a,h,x))\leq r$ (in other words, $B(a,rt)\cap[a,b]_\kappa \subseteq [a,h]_{r}$).
    \item $A\subset [a,h]_{r}$.
  \end{itemize}
\end{proposition}

\begin{corollary}\label{cor:coarse-hull}
  Let $X$ be a quasigeodesic coarse median space of rank at most $d$.
  For every $\kappa>0$ and $t>0$, there exists $r_t>0$, such that for all $r\geq r_t$, $a,b\in X$ there exists $h\in [a,b]_{L_1(r)}$, such that
  \begin{itemize}
    \item $\rho(a,h)\leq L_3(r,t)$ and
    \item $B(a,rt)\cap [a,b]_\kappa\subset [a,h]_{L_2(r)}$.
  \end{itemize}
\end{corollary}

\begin{proof}[Proof of Corollary \ref{cor:coarse-hull}]
  This readily follows from Proposition \ref{prop:coarse-hull}, by observing that we may choose $A$ to be a maximal $r$-separated subset of $B(a,rt)\cap[a,b]_\kappa$. Then any point $x\in B(a,rt)\cap[a,b]_\kappa$ is at most $r$-far from a point $a_x\in A$, hence the condition $A\subset[a,h]_r$ implies that
  \begin{align*}
    \m(a,h,x)  &\sim_{Kr+H(0)} \m(a,h,a_x) \sim_{r} a_x \sim_{r} x.
  \end{align*}
  Since we denoted $Kr+H(0)+r+r=L_2(r)$, the above reads $x\in[a,h]_{L_2(r)}$.
\end{proof}

\begin{proof}[Proof of Proposition \ref{prop:coarse-hull}]
  We proceed by contradiction: suppose that for some $\kappa$ and $t$ the statement is not true, i.e. there exists a sequence $0<r_1<r_2<\dots\in\R$ and for each $n\in\N$ there exist $a_n,b_n\in X$ and $A_n\subset[a_n,b_n]_\kappa\cap B(a_n,r_nt)$ with $\sep(A_n)\geq r_n$, such that for all $h\in [a_n,b_n]_{L_1(r_n)}$ with $\rho(a_n,h)\leq L_3(r_n,t)$ there exists $x\in A_n$ with $\rho(x,\m(a_n,h,x))>r_n$.

  It follows from \cite[Lemma 9.7]{Bowditch:2014aa} that we can assume that the cardinalities $|A_n|$ are uniformly bounded by a constant $p$ depending on $K,H,d,\kappa,t$.

  The next step is to argue that we can arrange that the distances from $a_n$ to $b_n$ are linear in $r_n$.

  \textbf{Claim.} There exist constants $\delta_1,\delta_2,\kappa_1\geq 0$ (depending only on $K,H,\kappa,t$ and $p$) and points $b_n'\in [a_n,b_n]_\lambda$, such that $\rho(a_n,b_n')\leq \delta_1 r_n+\delta_2$ and $A_n\subset [a_n,b_n']_{\kappa_1}$.

  The Claim follows from \cite[Lemma 9.6]{Bowditch:2014aa}, which says that in our situation there are constants $\zeta,\xi,\kappa'$ (depending only on $K,H,\kappa,t$ and $p$), and points $c_n,d_n\in X$, such that $A_n\subset [c_n,d_n]_{\kappa'}$ and $\diam(A_n\cup\{c_n,d_n\})\leq \zeta\diam(A_n)+\xi\leq2\zeta r_nt+\xi$. Since $A_n\subset B(a_n,r_nt)$, by the proof of that Lemma we can assume that $c_n=a_n$ for every $n$. Finally, we define $b_n'=\m(a_n,b_n,d_n)\in [a_n,b_n]_\lambda$ and check that
  $$
  b_n'=\m(a_n,b_n,d_n) \sim_{K(2\zeta r_nt+\xi)+H(0)} \m(d_n,b_n,d_n) = d_n \sim_{2\zeta r_nt+\xi} a_n,
  $$
  and for every $x\in A_n$ (so that $\m(a_n,b_n,x)\sim_\kappa x$ and $\m(a_n,d_n,x)\sim_{\kappa'} x$)
  \begin{align*}
    \m(a_n,b_n',x) &= \m(a_n,\m(a_n,b_n,d_n),x)\\
    &\sim_\gamma \m(\m(a_n,x,b_n),\m(a_n,x,d_n),a_n)\\
    &\sim_{K(\kappa+\kappa')+H(0)} \m(x,x,a_n) = x.
  \end{align*}
  So altogether, we put $\kappa_1=K(\kappa+\kappa')+H(0)+\gamma$, $\delta_1=2(K+1)\zeta t$ and $\delta_2=(K+1)\xi+H(0)$ and the Claim is proved.

  \smallskip

  We have now set up the situation so that we can conclude the proof by going to the asymptotic cone.

  % Consider now the sequence of rescaled metrics $\rho/r_n$, and choose the basepoints to be $a_n$. Then the asymptotic cone $(X_\infty,\rho_\infty,a)$ of $X$ with this data is a metric median algebra of rank at most $d$ (satisfying \red{L2}).
  % Moreover, the embedding results of Bowditch apply to this situation, see the penultimate paragraph in \cite[Section 1]{Bowditch:2014aa}.

  Let $(X_\infty,\rho_\infty,\m_\infty)$ be an asymptotic cone of $X$, with the sequence of scales $(r_n)$, basepoints $(a_n)$ and any non-principal ultrafilter on $\N$.

  The sequences $(a_n)$ and $(b_n')$ determine points $a,b\in X_\infty$ (with $\rho_\infty(a,b)\leq\delta_1$), and the intervals $[a_n,b_n']_{\kappa_1}$ converge to the interval $[a,b]$ in $X_\infty$. Also the sets $A_n$ converge to a (finite, $1$-separated) set $A\subset [a,b]\cap B(a,t)$.

  By Proposition \ref{prop:top-hull}, there exists $h\in[a,b]$, such that $A\subset B(a,t)\cap [a,b]\subset [a,h]$ and $\rho_\infty(a,h)\leq 3^dK^dt$. This implies that we have a sequence of points $h_n\in X$, eventually in $[a_n,b_n']_{r_n}$\footnote{Since $\m(a,b,h)=h$, we have $\rho(h_n,\m(a_n,b_n',h_n))/r_n\to 0$, hence the claim.}, such that $\lim\rho(a_n,h_n)/r_n \leq 3^dK^dt$, thus eventually $\rho(a_n,h_n)\leq 3^dK^dr_nt+r_n = L_3(r_n,t)$.

  Since $b_n'\in[a_n,b_n]_\lambda$, and $h_n\in[a_n,b_n']_{r_n}$, by Lemma \ref{lem:containment-int} we have that $h_n\in[a_n,b_n]_{L_1(r_n)}$. Hence, by our original assumption, there (eventually) exist points $x_n\in A_n$ with $\rho(x_n,\m(a_n,h_n,x_n))>r_n$.
  The sequence of $x_n$ yields a point $x\in A$, such that $\rho_\infty(x,\m_\infty(a,h,x)) \geq 1$. This point witnesses that $A\not\subset [a,h]$, which is a contradiction.
\end{proof}

% \begin{lemma}\label{lem:switch-coarse}
%   Let $X$ be a \red{coarse median\ldots}. For every $\kappa>0$ there exists $\kappa'>0$, such that whenever $a,b\in X$, $x\in [a,b]_\kappa$ and $y\in [a,x]_\kappa$, then $x \in [y,b]_{\kappa'}$.
% \end{lemma}

% \blue{
% \begin{proof}
%   For median, it is just Lemma \ref{lem:switch-in-int}. The proof is just a computation with $\m$ so it should hold uniformly in coarse median.
% \end{proof}
% }

\begin{lemma}\label{lem:coarse-switch-r}
  Let $X$ be a coarse median space. There exist constants $\alpha,\beta\geq0$ (depending only on the parameters of the coarse median structure), such that the following holds: let $a,b,h,m\in X$ and $r\geq0$ satisfy $m\in [a,h]_{L_2(r)}$, $h\in[a,b]_{L_1(r)}$. Then $p=\m(m,b,h)$ satisfies $\rho(h,p)\leq \alpha r+\beta$.
  % \begin{itemize}
  %   \item $p\in[m,b]_\kappa$, and
  %   \item $\rho(h,p)\leq \alpha r+\beta$.
  % \end{itemize}
\end{lemma}

\begin{proof}
  Note that the assumptions say that $m\sim_{L_2(r)}\mu(a,h,m)$ and $h\sim_{L_1(r)} \mu(a,b,h)$. We estimate
  \begin{align*}
  p &= \m(m,b,h) \sim_{KL_2(r)+H(0)} \m(\m(a,h,m),b,h)\\
  &\sim_{\gamma} \m(\m(b,h,h),\m(b,h,a),m)  \\
  &\sim_{KL_1(r)+H(0)} \m(h,h,m) = h.
  \end{align*}
  Altogether, $\rho(h,p)\leq K(L_1(r)+L_2(r))+2H(0)+\gamma$, which is a linear function of $r$.
\end{proof}

\begin{theorem}\label{thm:main}
  Let $X$ be a uniformly locally finite, quasigeodesic, at most exponential growth, coarse median space of finite rank. Then $X$ has property A.
\end{theorem}

\begin{proof}
  The proof follows the idea of our proof for CAT(0) cube complexes, which relies on Proposition \ref{prop:criterion}. We shall verify its assumptions. Let $\alpha,\beta>0$ be the constants from Lemma \ref{lem:coarse-switch-r} and fix a basepoint $x_0\in X$. 

  We now apply Corollary \ref{cor:coarse-hull} for $\kappa=\lambda$ and all $t\in \N$ to obtain a sequence $r_t\in\N$, so that the conclusion of the Corollary holds. Furthermore, we can choose the $r_t$ inductively to arrange that the sequence $\N\ni t\mapsto l_t=\frac{tr_t-H(0)}{3K}$ is increasing.

  For a moment, fix $x\in X$, $t\in\N$ and $k\in\{1,\dots, 3l_t\}$. For every $y\in B(x,k)$, Corollary \ref{cor:coarse-hull} applied for $a=y$, $b=x_0$ and $r=r_t$ produces for us a point $h_y\in [y,x_0]_{L_1(r_t)}$. We collect these points into the set
  $$
  S(x,k,l_t) = \left\{h_y\in X\mid y\in B(x,k)\right\}.
  $$
  Loosely speaking, the set $S(x,k,l_t)$ contains one point associated to each $y\in B(x,k)$, which should be thought of as being ``$tr_t$-deep'' inside the interval $[y,x_0]_\lambda$.

  Defined like this, the condition (ii) of Proposition \ref{prop:criterion} is automatic, (i) follows from the first bullet of Corollary \ref{cor:coarse-hull}, and finally (iii) requires some checking:

  Take $y\in B(x,k)$, with the notation as above. Denote $m_y=\m(x,y,x_0)$. Then $m_y\in[y,x_0]_\lambda$ and 
  $$
  \rho(y,m_y)\leq K\rho(x,y)+H(0)\leq K\cdot 3l_t + H(0)=tr_t.
  $$
  Thus the second bullet of Corollary \ref{cor:coarse-hull} implies that $m_y\in [y,h_y]_{L_2(r_t)}$. Since we also know that $h_y\in[y,x_0]_{L_1(r_t)}$, Lemma \ref{lem:coarse-switch-r} now implies that the point $p_y=\m(m_y,x_0,h_y)\in [m_y,x_0]_\lambda$ satisfies $\rho(h_y,p_y)\leq \alpha r_t+\beta$. As $m_y=\m(x,y,x_0)\in[x,x_0]_\lambda$, Lemma \ref{lem:containment-int} implies $p_y\in [x,x_0]_{L_1(\lambda)}$.

  To summarise, for each $h_y\in S(x,k,l_t)$ we can associate a point $p_y\in [x,x_0]_{L_1(\lambda)}$ satisfying $\rho(h_y,p_y)\leq \alpha r_t+\beta$, and consequently also
  $$
  \rho(x,p_y) \leq \rho(x,y)+\rho(y,h_y)+\rho(h_y,p_y) \leq 3l_t + 3^dK^dtr_t + r_t + \alpha r_t+\beta,
  $$
  which is clearly depends linearly on $l_t$. Hence, by \cite[Proposition 9.8]{Bowditch:2014aa}, the number of the possible points $p_y$ is bounded by $P(l_t)$ for some polynomial $P$ (depending only on $H$, $K$, $d$, and uniform local finiteness of $X$). Since we assume at most exponential growth of $X$, it follows that the cardinality of $S(x,k,l_t)$ is at most $P(l_t)c'c^{r_t}$ for some constants $c,c'\geq 1$. Finally, as $l_t\to\infty$ means by definition also $t\to\infty$, it is easy to see that also $r_t/l_t\to 0$, thus
  $\left(P(l_t)c'c^{r_t}\right)^{1/l_t} \to 1$. This finishes proving the condition (iii) of Proposition \ref{prop:criterion} and we are done.
\end{proof}

%\bibliographystyle{alpha}
%\bibliography{cm-a}

\newcommand{\etalchar}[1]{$^{#1}$}

\end{document}